\documentclass[11pt,reqno,oneside]{amsart}
\usepackage[utf8]{inputenc}
\usepackage{amsmath,amsthm,amssymb}
\usepackage{colonequals}
\usepackage{graphics}
\usepackage[pagebackref=true]{hyperref}
\usepackage[usenames, dvipsnames]{xcolor}
\usepackage{mathtools}
\mathtoolsset{showonlyrefs}
\usepackage[square,sort,comma,numbers]{natbib}

\definecolor{darkblue}{rgb}{0.0,0.0,0.84}
\hypersetup{colorlinks,breaklinks,
  linkcolor=darkblue,urlcolor=darkblue,
anchorcolor=darkblue,citecolor=darkblue}

\theoremstyle{plain}
\newtheorem{theorem}{Theorem}
\newtheorem*{theorem*}{Theorem}
\newtheorem{lemma}[theorem]{Lemma}
\newtheorem{proposition}[theorem]{Proposition}
\newtheorem*{proposition*}{Proposition}

\newtheorem*{corollary*}{Corollary}

\theoremstyle{definition}
\newtheorem{remark}[theorem]{Remark}

\newtheorem{example}[theorem]{Example}

\renewcommand{\Im}{\operatorname{Im}}
\renewcommand{\Re}{\operatorname{Re}}

\DeclareMathOperator{\Tr}{Tr}
\newcommand{\Z}{\mathbb{Z}}
\newcommand{\sbar}{\overline{s}}
\newcommand{\calO}{\mathcal{O}}
\newcommand{\Zeven}{Z_D^{\textup{even}}}
\newcommand{\Zodd}{Z_D^{\textup{odd}}}
\newcommand{\Zboth}{Z_D}
\newcommand{\chid}{\chi_{4D}}

\title{The Fibonacci Zeta Function and Continuation}

\author[Assaf]{Eran Assaf}
\thanks{EA was supported by Simons Collaboration Grant (550029, to Voight).}

\author[Kuan]{Chan Ieong Kuan}
\thanks{CIK was supported in part by NSFC (No.\ 11901585).}

\author[Lowry-Duda]{David Lowry-Duda}
\thanks{DLD was supported by the Simons Collaboration in Arithmetic Geometry, Number
Theory, and Computation via the Simons Foundation grant 546235.}

\author[Walker]{Alexander Walker}
\thanks{AW was supported by the Additional Funding Programme for Mathematical Sciences,
delivered by EPSRC (EP/V521917/1) and the Heilbronn Institute for Mathematical Research.}

\date{Last compiled: \today}

\begin{document}

\begin{abstract}
  We introduce a family of Dirichlet series associated to real quadratic number fields
  that generalize the ordinary Fibonacci zeta function $\sum F(n)^{-s}$, where $F(n)$ denotes the
  $n$th Fibonacci number.
  We then give three different methods of meromorphic continuation to $\mathbb{C}$.
  Two are purely analytic and classical, while the third uses shifted convolutions and
  modular forms.
\end{abstract}

\maketitle

\section{Introduction and Context}%
\label{sec:intro}

The Dirichlet series $\sum_{n \geq 1} a(n)/n^s$ attached to a sequence $\{ a(n)
  \}_{n \geq 1}$ of polynomial growth encodes arithmetic qualities of the
sequence within the analytic properties in $s$.
A generic Dirichlet series will converge in some half-plane $\Re s > \sigma$.
For distinguished sequences, the associated Dirichlet series will have meromorphic
continuation to a larger region, possibly all of $\mathbb{C}$.
This occurs most famously with the Riemann zeta function, when $a(n) \equiv 1$,
as well as Dirichlet $L$-functions, when $a(n) = \chi(n)$ for a Dirichlet
character $\chi$.

In~\cite{egami1999curious, navas2001fibonacci}, Egami and Navas independently studied the
zeta function $Z_{\mathrm{Fib}}(s)$ associated to the indicator function of the Fibonacci numbers
$\{ 1, 1, 2, 3, \ldots \}$,
given by
\begin{equation}
  Z_{\mathrm{Fib}}(s)
  :=
  \sum_{n \geq 1} \frac{1}{F(n)^s}
  =
  1 + 1 + \frac{1}{2^s} + \frac{1}{3^s} + \cdots
\end{equation}
As $F(n) \asymp \varphi^n$ where $\varphi = \frac{1 + \sqrt{5}}{2}$, it is clear that
$Z_{\mathrm{Fib}}(s)$ converges absolutely for $\Re s > 0$.
Egami and Navas both show that $Z_{\mathrm{Fib}}$ has meromorphic continuation to all
of $\mathbb{C}$, with a half-lattice of poles in the region $\Re s \leq 0$.

\begin{theorem*}[Egami~\cite{egami1999curious} and Navas~\cite{navas2001fibonacci}, independently]
  The Fibonacci zeta function $Z_{\mathrm{Fib}}(s)$ has meromorphic continuation to all of
  $\mathbb{C}$, given explicitly by
  \begin{equation}\label{eq:egami_navas}
    Z_{\mathrm{Fib}}(s)
    =
    5^{\frac{s}{2}}
    \sum_{k \geq 0}
    {-s \choose k}
    \frac{1}{\varphi^{s + 2k} + (-1)^{k+1}}.
  \end{equation}
  The function $Z_{\mathrm{Fib}}(s)$ has simple poles at $s = -2k + \frac{(2n + k)\pi i}{\log \varphi}$
  for $k, n \in \mathbb{Z}$, $k \geq 0$, and is otherwise holomorphic.
\end{theorem*}

Kamano~\cite{kamano2013analytic} proved an analogous continuation for the zeta function associated to the Lucas sequence $\{ 2, 1, 3, 4, 7, \ldots\}$.
Several mathematicians have since studied zeta functions related to the Fibonacci zeta
function.
The recent work of Smajlovi\'{c}, \v{S}abanac, and
\v{S}\'{c}eta~\cite{smajlovic2022tribonacci} on \emph{Tribonacci} zeta functions includes a
good overview of recent papers in their introduction.

In this paper, we introduce a more general family of zeta functions with similar
properties.
For a positive square-free integer $D$, let $\calO_D$ denote the
ring of integers in $\mathbb{Q}(\sqrt{D}) \subset \mathbb{R}$.
Let $\varepsilon > 1$ denote the fundamental unit in $\calO_D$.
We define the $\calO_D$ Lucas and Fibonacci sequences in terms of
traces of the fundamental unit,
\begin{align}
  L_{D}(n) & := \Tr_{\calO_D}(\varepsilon^n)             %
  \label{align:lucas_def}%
  \\
  F_{D}(n) & := \Tr_{\calO_D}(\varepsilon^n / \sqrt{q}), %
  \label{align:fib_def}%
\end{align}
where $q = D$ if $D \equiv 1 \bmod 4$ and otherwise $q = 4D$. (Thus $F(n) = F_5(n)$.)

In this article, we give three distinct proofs of the following theorem.

\begin{theorem}\label{thm:main}
  Suppose that the fundamental unit $\varepsilon$ satisfies $N(\varepsilon) = -1$.
  Then the zeta function associated to the odd-indexed terms and the zeta
  function associated to the even-indexed terms of the $\calO_D$ Fibonacci
  sequence,
  \begin{align}
    \Zodd(s)
     & :=
    \sum_{n \geq 1}
    \frac{1}{F_D(2n - 1)^s}
    =
    \sum_{n \geq 1}
    \frac{1}{\big(\Tr_{\calO_D}(\varepsilon^{2n - 1} / \sqrt{q}) \big)^s},
    \\
    \Zeven(s)
     & :=
    \sum_{n \geq 1}
    \frac{1}{F_D(2n)^s}
    =
    \sum_{n \geq 1}
    \frac{1}{\big(\Tr_{\calO_D}(\varepsilon^{2n} / \sqrt{q}) \big)^s},
  \end{align}
  have meromorphic continuation to $\mathbb{C}$, with a half-lattice of poles in the region
  $\Re s \leq 0$.
\end{theorem}

As an immediate corollary, this implies that the combined zeta function $\Zboth(s) = \sum_{n \geq 1}
  F_D(n)^{-s}$ has meromorphic continuation to $\mathbb{C}$.

\begin{remark}
  The fundamental unit $\varepsilon$ satisfies $N(\varepsilon) = -1$ if and
  only if the class group and narrow class group of $\mathbb{Q}(\sqrt{D})$
  coincide. This assumption is not necessary, though it does streamline our
  presentation.
  This assumption also restricts $D$ to the congruence classes $D \equiv
    1, 2 \bmod 4$, as a unit of norm $-1$ when $D \equiv 3 \bmod 4$ would imply
  that $3$ is the sum of two squares mod $4$.
  When $N(\varepsilon) =1$, $\Zboth(s)$ may be treated using the
  techniques presented for $\Zeven(s)$.
\end{remark}

In \S\ref{sec:binomial}, we give a basic continuation using the binomial series.
This method of continuation is well-known in the literature and a version seems to appear
in every article on the Fibonacci zeta function.

In \S\ref{sec:poisson}, we give an alternate continuation using Poisson summation. We treat $\Zodd(s)$ and $\Zeven(s)$ separately, owing to additional complications in the latter case.

Proofs similar to the binomial method and the Poisson method appeared in solutions to
problem 10486 in the American Mathematical Monthly~\cite{silverman1999zeta};
Silverman proposed that readers study the continuation of
\begin{equation} \label{eq:AMM-series}
  \sum_{n \in \mathbb{Z}}
  \frac{1}{(a \alpha^{n} + b \alpha^{-n})^s},
  \qquad (a, b > 0 \, ; \, \alpha > 1).
\end{equation}
This series closely resembles the two-sided odd-indexed Fibonacci zeta
function, $\sum_{n \in \mathbb{Z}} F(2n - 1)^{-s}$.

The American Mathematical Monthly published two solutions: a solution based on binomial
expansion (due to Bradley) and a solution based on Poisson summation (due to Darling).
Silverman notes via private correspondence that he provided a proof using
binomial expansion when he first proposed the problem to the American
Mathematical Monthly in 1995.
Our treatment of $\Zodd(s)$ in~\S\ref{sec:poisson} resembles Darling's proof.
Considering $\Zeven(s)$, which corresponds to a version
of~\eqref{eq:AMM-series} with $b<0$, is more intricate and requires additional
techniques.

Finally, in \S\ref{sec:modular}, we show how to obtain the meromorphic continuation of
$\Zboth$ by interpreting the problem through modular forms.
This is a completely new proof.
Our approach is based on the observation that an integer $n$ is an $\calO_D$
Fibonacci number if and only if there is an integer solution for $X$ in the equation $X^2
  = q n^2 \pm 4$, where $q = D$ or $q = 4D$, depending on whether $d \equiv 1 \bmod 4$.
This connection to Pell-type equations is proved in Proposition~\ref{prop:pfib}.

The authors were surprised that one can prove continuation of the Fibonacci
zeta function from modular forms.
But as the proofs are distinct from the rest of the article and rather
technical, we defer the finer details to our forthcoming sequel
work~\cite{akldwFibonacciModular}.

\section*{Acknowledgements}

We would like to thank John Voight for initial guidance on a talk that grew
into this paper.
We also thank Andy Booker, Jeff Hoffstein, Min Lee, Will Sawin, and Joseph
Silverman for helpful discussion.

\section{Fibonacci and Lucas sequences from
  \texorpdfstring{$\calO_D$}{O(Q(sqrt(D)))}}%
\label{sec:pfibluc}

We now describe in greater detail the $\calO_D$ Fibonacci
and Lucas numbers defined in~\eqref{align:lucas_def}
and~\eqref{align:fib_def}.
These sequences generalize the typical Fibonacci and Lucas sequences.

\begin{example}
  Consider $D = 5$.
  Then $\varphi = \frac{1 + \sqrt{5}}{2} = \varepsilon$ is the fundamental unit in
  $\calO_{5}$ and $-1/\varphi$ is its conjugate.
  We see that $\Tr(\varepsilon^n) = (\varphi^n + (-1/\varphi)^{n})$ and
  $\Tr(\varepsilon^n/\sqrt{5}) = (\varphi^n - (-1/\varphi)^{n})/\sqrt{5}$ are the typical
  Lucas and Fibonacci numbers, respectively, given by their Binet formulas.
  Stated differently, the $\calO_{5}$ Fibonacci and Lucas sequences are the
  \emph{standard} Fibonacci and Lucas sequences.
\end{example}

More generally, given a real quadratic field $\mathbb{Q}(\sqrt{D})$ with
squarefree, positive $D$, ring of integers $\calO_D$, and fundamental unit
$\varepsilon$, the $\calO_D$ Fibonacci and Lucas sequences correspond to solutions
to the quadratic recurrence
\begin{equation}
  a(n + 2) = \Tr(\varepsilon) a(n+1) - N(\varepsilon)a(n)
\end{equation}
with different initial conditions.
This is because $F_D(n)$ and $L_D(n)$ are sums of powers of $\varepsilon$
and $\overline{\varepsilon}$, which are the roots of the characteristic equation
$X^2 - \Tr(\varepsilon) X + N(\varepsilon) = 0$.

\begin{example}
  Consider $D = 3$.
  The fundamental unit is $\varepsilon = 2 + \sqrt{3}$, which has trace $4$ and
  norm $1$.
  Then $F_3(n)$ and $L_3(n)$ satisfy the recurrence $a(n+2) = 4 a(n+1) - a(n)$.
  Starting with the term when $n = 0$,
  the $\calO_3$ Fibonacci sequence begins $\{ 0, 1, 4, 15, 56, 209, \ldots \}$
  and the $\calO_3$ Lucas sequence begins $\{ 2, 4, 14, 52, 194, 724, \ldots \}$.
\end{example}

\begin{example}
  Consider $D = 10$.
  The fundamental unit is $\varepsilon = 3 + \sqrt{10}$, which has trace $6$ and
  norm $-1$.
  Then $F_{10}(n)$ and $L_{10}(n)$ satisfy the recurrence $a(n+2) = 6 a(n+1) + a(n)$.
  The $\calO_{10}$ Fibonacci sequence begins $\{ 0, 1, 6, 37, 228, \ldots\}$
  and the $\calO_{10}$ Lucas sequence begins $\{2, 6, 38, 234, 1442, \ldots \}$.
\end{example}

%

\begin{proposition}\label{prop:pfib}
  With the definitions above, an integer $n$ is an $\calO_D$ Fibonacci number
  if and only if there is an integer solution in $X$ to the equation
  \begin{equation}\label{eq:proppfib}
    X^2 = qn^2 \pm 4,
  \end{equation}
  where $q = D$ if $D \equiv 1 \bmod 4$ and otherwise $q = 4D$.
  If in addition $N(\varepsilon) = -1$, then $n$ is an odd-indexed $\calO_D$
  Fibonacci number if and only if there is an integer solution in $X$ to the equation
  \begin{equation}
    X^2 = qn^2 - 4.
  \end{equation}
\end{proposition}

\begin{proof}
  Each number $\alpha$ in the ring of integers $\calO_D$ can be written uniquely
  as
  \begin{equation}
    \alpha = m + n \frac{q + \sqrt{q}}{2},
  \end{equation}
  where $m$ and $n$ are integers.
  (This differs from the standard integral basis; we refer the reader
  to~\cite[Chapter~2]{marcus2018numberfields} for an excellent general
  reference on the topics presented here.)
  The element $\alpha$ is a unit if and only if $N(\alpha) = \pm 1$, or equivalently if and only if
  \begin{equation}\label{eq:prop_equiv}
    b^2 = qn^2 \pm 4
  \end{equation}
  where $b = 2m + qn$.

  Suppose that $b$ and $n$ are positive integer solutions
  to~\eqref{eq:prop_equiv}, witnessing that $\alpha$ a unit.
  Then $\alpha = \pm \varepsilon^r$ for some integer $r$.
  Rearranging, we find that
  \begin{align}
    \pm \alpha = \frac{b + n \sqrt{q}}{2} = \varepsilon^r
     & =
    \frac{1}{2}
    \Bigl[
      (\varepsilon^r + \overline{\varepsilon}^r)
      +
      \frac{\varepsilon^r - \overline{\varepsilon}^r}{\sqrt{q}} \sqrt{q}
      \Bigr] \!
    =
    \frac{1}{2}
    \bigl[
      L_D(r) + F_D(r) \sqrt{q}
      \bigr].
  \end{align}
  Thus if $(b, n)$ is a positive solution to~\eqref{eq:prop_equiv}, then $b = L_D(r)$ and
  $n = F_D(r)$ for some $r$.

  Conversely, if $n = F_D(r)$ for some $r$, then taking norms in
  $\varepsilon^r = \frac{1}{2} \bigl[ L_D(r) + F_D(r) \sqrt{q} \bigr]$
  shows that
  \begin{equation}
    L_D(r)^2 - q F_D(r)^2 = 4 N(\varepsilon^r),
  \end{equation}
  hence $(b,n) = (L_D(r), F_D(r))$ is a solution to~\eqref{eq:prop_equiv}.

  The case when $N(\varepsilon) = -1$ follows by tracking signs of $N(\varepsilon^r) =
    (-1)^r$.
\end{proof}

\begin{remark}
  When $q = 4D$, equation~\eqref{eq:proppfib} becomes $X^2 = 4Dn^2 \pm 4$.
  Any solution in $X$ must necessarily be even, say $X = 2Y$.
  Then equation~\eqref{eq:proppfib} reduces to
  \begin{equation}
    Y^2 = Dn^2 \pm 1.
  \end{equation}
  This minor parity annoyance appears in \S\ref{sec:modular}.
  Up to this parity problem, a similar result can be shown for $\calO_D$ Lucas
  numbers.
\end{remark}

\section{Continuation via Binomial Expansion}%
\label{sec:binomial}

Our first proof of Theorem~\ref{thm:main} involves a direct application of the binomial series.
As $N(\varepsilon) = -1$ by assumption, we know that the Galois conjugate
$\overline{\varepsilon}$ equals $-1/\varepsilon$.
Hence
\begin{equation}
  F_D(2n - 1) = \frac{\varepsilon^{2n-1} + \varepsilon^{1 - 2n}}{\sqrt{q}}
  \quad \text{and} \quad
  F_D(2n) = \frac{\varepsilon^{2n} - \varepsilon^{-2n}}{\sqrt{q}}.
\end{equation}
For $\Re s > 0$, we substitute this into $\Zodd$ to produce
\begin{equation} \label{eq:Z_d_odd_sum}
  \Zodd(s)
  =
  \sum_{n \geq 1}
  \frac{q^{s/2}}
  {(\varepsilon^{2n - 1} + \varepsilon^{1 - 2n})^s}
  =
  q^{\frac{s}{2}}
  \sum_{n \geq 1}
  \varepsilon^{(2n - 1)s}
  (\varepsilon^{4n - 2} + 1)^{-s}.
\end{equation}
After expanding $(\varepsilon^{4n - 2} + 1)^{-s}$ using the binomial series and
collecting powers of $\varepsilon^{-2s-4k}$ in a geometric series (where we use
that $\varepsilon > 1$), we obtain
\begin{equation}
  \Zodd(s)
  =
  q^{\frac{s}{2}}
  \sum_{n \geq 1}
  \varepsilon^{(2n - 1)s}
  \sum_{k \geq 0}
  {-s \choose k}
  (\varepsilon^{4n - 2})^{-s-k}
  =
  q^{\frac{s}{2}}
  \sum_{k \geq 0}
  {-s \choose k}
  \frac{\varepsilon^{s + 2k}}{\varepsilon^{2s + 4k} - 1}.
\end{equation}
This sum converges absolutely for all $s$ away from poles in the summands and therefore gives a meromorphic continuation of $\Zodd(s)$ to all $s \in \mathbb{C}$. Note that $\Zodd(s)$ admits simple poles on lines
$\Re s = -2k$ for $k \in \mathbb{Z}_{\geq 0}$.
Similarly, for $\Re s > 0$, we compute
\begin{equation}\label{eq:Z_d_even_binomial}
  \Zeven(s)
  =
  \sum_{n \geq 1}
  \frac{q^{s/2}}
  {(\varepsilon^{2n} - \varepsilon^{-2n})^s}
  =
  q^{\frac{s}{2}}
  \sum_{k \geq 0}
  {-s \choose k}
  \frac{(-1)^k}{\varepsilon^{2s + 4k} - 1}.
\end{equation}
This gives meromorphic continuation and shows that $Z_D^{\textup{even}}$ has
simple poles in the same locations as $\Zodd$.
We collect this in the following theorem.

\begin{theorem} \label{thm:binomial-expression}
  The odd-indexed and even-indexed $\calO_D$ Fibonacci zeta functions
  $\Zodd(s)$ and $\Zeven(s)$ meromorphically continue to $\mathbb{C}$ via
  \begin{equation}
    \Zodd(s)
    =
    q^{\frac{s}{2}}
    \sum_{k \geq 0}
    {-s \choose k}
    \frac{\varepsilon^{s + 2k}}{\varepsilon^{2s + 4k} - 1},
    \quad
    \Zeven(s)
    =
    q^{\frac{s}{2}}
    \sum_{k \geq 0}
    {-s \choose k}
    \frac{(-1)^k}{\varepsilon^{2s + 4k} - 1}.
  \end{equation}
  Both have simple poles at $s = -2k + \frac{\pi m}{\log \varepsilon} i$ for $k
    \in \mathbb{Z}_{\geq 0}$ and $m \in \mathbb{Z}$.
\end{theorem}

Adding $\Zodd$ and $\Zeven$ together recovers a generalization of the
continuation of Egami and Navas in~\eqref{eq:egami_navas},
\begin{equation}\label{eq:generalized_egami_navas}
  \Zboth(s)
  =
  q^{\frac{s}{2}}
  \sum_{k \geq 0}
  {-s \choose k}
  \frac{\varepsilon^{s + 2k} + (-1)^k}{\varepsilon^{2s + 4k} - 1}
  =
  q^{\frac{s}{2}}
  \sum_{k \geq 0}
  {-s \choose k}
  \frac{1}{\varepsilon^{s + 2k} + (-1)^{k+1}},
\end{equation}
where we've used that $\varepsilon^{2s + 4k} - 1 = (\varepsilon^{s + 2k} -
  1)(\varepsilon^{s + 2k} + 1)$ to simplify.

\begin{remark}
  As~\eqref{eq:generalized_egami_navas} only has simple poles at $s = -2k + i
    (2n + k) \pi / \log \varepsilon$, this shows that half of the poles of
  $Z_D^{\textup{odd}}(s)$ perfectly cancel half of the poles of
  $Z_D^{\textup{even}}(s)$.
  This may also be seen through direct computation of residues.
\end{remark}

\section{Continuation via Poisson Summation}%
\label{sec:poisson}

In this section, we describe how the meromorphic continuations of the
Fibonacci zeta functions $\Zodd(s)$ and $\Zeven(s)$ can be
obtained using Poisson summation instead of binomial series manipulations.
We separate our discussion and consider $\Zodd$ first.

\subsection{Odd-indexed Fibonacci zeta function}

Beginning with the explicit sum for $\Zodd(s)$ from~\eqref{eq:Z_d_odd_sum}, we write
\begin{equation}
  \Zodd(s)
  = \frac{q^{s/2}}{2}
  \sum_{\substack{n \in \mathbb{Z} \\ n \equiv 1 \bmod 2}}
  f_s(n),
  \qquad
  f_s(n):=(\varepsilon^{n} + \varepsilon^{-n})^{-s},
\end{equation}
in which the extension from $n > 0$ to $n \in \mathbb{Z}$ follows by symmetry.
Once $\Zodd(s)$ is recognized as a complete sum over a congruence class, we apply Poisson summation to write
\begin{align} \label{eq:Z-odd-Poisson}
  \Zodd(s) = \tfrac{1}{2} q^{s/2} \cdot
  \tfrac{1}{2} \sum_{m \in \mathbb{Z}} \widehat{f_s}(\tfrac{m}{2}) e(\tfrac{m}{2}),
\end{align}
in which $e(z) := e^{2\pi i z}$ and $\widehat{f_s}(x)$ is the Fourier transform
\begin{align} \label{eq:Fourier-transform-definition}
  \widehat{f_s}(m) := \int_{-\infty}^\infty f_s(x) e(-mx) \, dx.
\end{align}
Poisson summation is justified by absolute convergence of the relevant series and integrals when $\Re s > 0$.

The Fourier coefficients $\widehat{f_s}(m)$ have an explicit description in
terms of the classical beta function $\mathrm{B}(x,y) := \Gamma(x)\Gamma(y)/ \Gamma(x+y)$.

\begin{lemma} \label{lem:Fourier-coefficient-odd}
  For $\Re s > 0$, we have
  \[
    \widehat{f_s}(m) =
    \frac{1}{2 \log \varepsilon} \mathrm{B}(\tfrac{s}{2} + \tfrac{\pi i m}{\log \varepsilon},
    \tfrac{s}{2} - \tfrac{\pi i m}{\log \varepsilon})
    =
    \frac{\Gamma(\frac{s}{2} + \frac{\pi i m}{\log \varepsilon})
      \Gamma(\frac{s}{2} - \frac{\pi i m}{\log \varepsilon})}
    {2 \Gamma(s) \log \varepsilon}.
  \]
\end{lemma}

\begin{proof}
  By~\eqref{eq:Fourier-transform-definition} and the standard integral representation of $\Gamma(s)$, we have
  \begin{align}
    \Gamma(s) \widehat{f_s}(m)
     & = \int_0^\infty \int_{-\infty}^\infty t^s e^{-t} \cdot
    (\varepsilon^x + \varepsilon^{-x})^{-s} e(-mx) \, \frac{dx dt}{t}
  \end{align}
  Introducing the variable $z = \varepsilon^x$ and then mapping $t \mapsto (z+z^{-1})t$ produces
  \[
    \Gamma(s) \widehat{f_s}(m)
    =\int_0^\infty \int_{0}^\infty t^s e^{-t(z+1/z)}
    z^{-\frac{2\pi i m}{\log \varepsilon}} \, \frac{dz dt}{zt \log \varepsilon}.
  \]
  Finally, we change variable again, introducing $u=tz$ and $v=t/z$. (This gives $t= \sqrt{uv}$ and $z = \sqrt{u/v}$.)
  Since the Jacobian of this transformation is $1/2v$, we obtain
  \begin{align}
    \Gamma(s) \widehat{f_s}(m)
     & =\int_0^\infty \int_{0}^\infty u^{\frac{s}{2}} v^{\frac{s}{2}}  e^{-u-v}
    u^{-\frac{\pi i m}{\log \varepsilon}} v^{\frac{\pi i m}{\log \varepsilon}} \, \frac{dudv}{2 uv \log \varepsilon} \\
     & = \frac{1}{2 \log \varepsilon} \Gamma(\tfrac{s}{2} + \tfrac{\pi i m}{\log \varepsilon})
    \Gamma(\tfrac{s}{2} - \tfrac{\pi i m}{\log \varepsilon}),
  \end{align}
  which completes the proof.
\end{proof}

By applying Lemma~\ref{lem:Fourier-coefficient-odd}
to~\eqref{eq:Z-odd-Poisson}, we produce the following meromorphic continuation
of $\Zodd$.

\begin{theorem} \label{thm:continuation-Poisson-odd}
  For $s \in \mathbb{C}$ away from the poles of the summands,
  \[
    \Zodd(s) = \frac{q^{s/2}}{8 \Gamma(s) \log \varepsilon}
    \sum_{m \in \mathbb{Z}} (-1)^m
    \Gamma\Big(\frac{s}{2} + \frac{\pi i m}{2\log \varepsilon}\Big)
    \Gamma\Big(\frac{s}{2} - \frac{\pi i m}{2\log \varepsilon}\Big),
  \]
  Thus $\Zodd(s)$ has meromorphic continuation to $s \in \mathbb{C}$, with simple poles at $s = - 2k + \frac{\pi i m}{\log \varepsilon}$
  for $m \in \mathbb{Z}$ and integral $k \geq 0$.
\end{theorem}

The apparent double poles at non-positive even integers from the $m=0$ term are canceled by
$1/\Gamma(s)$, becoming simple poles as claimed.
Note that the convergence of the $m$-sum for any $s$ away
from poles follows from Stirling's approximation.

In addition, $1/\Gamma(s)$ induces a family of trivial zeros of $\Zodd(s)$ at
negative odd integers, which is not so obvious from
Theorem~\ref{thm:binomial-expression}.

\subsection{Even-Indexed Fibonacci Zeta Function}

Given the success of Poisson summation in continuing $\Zodd(s)$, we might wonder if this technique provides a meromorphic continuation in the even-indexed case. One immediate challenge is that the series
\begin{equation}\label{eq:zeven_def}
  \Zeven(s)
  := q^{\frac{s}{2}} \sum_{n \geq 1} \big( \varepsilon^{2n} - \varepsilon^{-2n} \big)^{-s}
\end{equation}
is difficult to relate to a sum over $n \in \mathbb{Z}$, in part because the corresponding $n=0$ term is formally $0^{-s}$.
A similar issue arises when attempting to understand the Riemann zeta function
via Poisson summation, as considered by Mordell in~\cite{mordell29}.
We take inspiration from Mordell and overcome this obstacle by applying Poisson
summation to a regularized sum.

A second complication arises from the fact that the summands of $\Zeven$, $(\varepsilon^{2n} - \varepsilon^{-2n})^{-s}$, are not invariant under $n \mapsto -n$ (unlike those of $\Zodd$). We overcome this obstacle by applying the following `truncated' Poisson summation formula.

\begin{lemma}[cf.~\cite{mordell29}] \label{lem:Mordell}
  Let $f(x)$ be a $C^1$ function on $[a,b]$.
  Suppose that $f''(x)$ exists on $[a,b]$ and that the integrals
  \[
    \int_a^b f(x) dx \quad \text{and} \quad
    \int_a^b \vert f''(x) \vert dx
  \]
  converge. If $a$ (respectively $b$) is infinite, suppose also that $f(x) \to
    0$, $f'(x) \to 0$ as $x$ tends to $-\infty$ (respectively $\infty$).
  Then
  \begin{align} \label{eq:partial-poisson}
    \sum_{a \leq n \leq b} f(n)
    = \lim_{M \to \infty} \sum_{\substack{ m \in \mathbb{Z} \\ \lvert m \rvert \leq M}}
    \int_a^b f(x) e^{2\pi i m x} dx,
  \end{align}
  where the endpoints $f(a), f(b)$ are counted with multiplicity $\frac{1}{2}$
  when $a,b \in \mathbb{Z}$.
\end{lemma}

We consider the function $f(x) = (\varepsilon^{2x} - \varepsilon^{-2x})^{-s} -
  (4 x\log \varepsilon)^{-s}$, which grows as $O_s(x^{2-\Re s})$ as $x \to 0^+$ and
as $O_s(x^{-\Re s})$ as $x \to \infty$.
Initially we restrict attention to $s$ with $\Re s \in (1, 2)$, where $f(x)$
decays nicely as both $x \to 0^+$ and $x \to \infty$.

For $s$ with $\Re s \in (1, 2)$, we apply Lemma~\ref{lem:Mordell} to $f(x)$
over $[0,\infty]$ (corresponding to a one-sided Fourier transform) to obtain
\[
  \Zeven(s)
  = \frac{q^{s/2}\zeta(s)}{(4 \log \varepsilon)^{s}}
  + q^{\frac{s}{2}}\sum_{m \in \mathbb{Z}}
  \int_0^\infty \!\!\! \big(
  (\varepsilon^{2x} - \varepsilon^{-2x})^{-s} - (4 x\log \varepsilon)^{-s}
  \big) e^{2\pi i m x} dx.
\]
For $m \neq 0$, we have (cf.~\cite[(equation~9.1, equation~10.1, and p.~291)]{mordell29}) the integral formula
\begin{align*}
   & \int_0^\infty \!\! \big(
  (\varepsilon^{2x} - \varepsilon^{-2x})^{-s} - (4 x\log \varepsilon)^{-s}
  \big) e^{2\pi i m x} dx     \\
   & \qquad
  = \frac{\Gamma(1-s) \Gamma(\frac{s}{2} - \frac{\pi i m}{2 \log \varepsilon})}
  {4 \log \varepsilon \Gamma(1-\frac{s}{2} - \frac{\pi i m}{2 \log \varepsilon})}
  - \frac{\Gamma(1-s)e^{\frac{\pi i}{2}(1-s) \mathrm{sgn}(m)}}{(2\pi \vert m \vert)^{1-s}(4 \log \varepsilon)^s}.
\end{align*}
(In fact, this is well-behaved for $\Re s \in (0, 3)$, but we
don't yet take advantage of this.)
When $m = 0$, we introduce $z = \varepsilon^{-2x}$; the constant phase then equals
\begin{align}
  \begin{split} \label{eq:constant-Fourier-coefficient}
     & \int_0^1 \big((1/z-z)^{-s} - (-2 \log z)^{-s}\big) \frac{dz}{2 z\log \varepsilon}   \\
     & \qquad = \frac{z^s {}_2F_1(s,\frac{s}{2}, \frac{s}{2}+1, z^2)}{2s \log \varepsilon}
    + \frac{(-2\log z)^{1-s}}{4(1-s) \log \varepsilon}\bigg\vert_0^1,
  \end{split}
\end{align}
in which we've used that $\frac{d}{dz}(z^s {}_2F_1(s,\frac{s}{2}, \frac{s}{2}+1, z^2)) = s z^{s-1} (1-z^2)^{-s}$.
Since $\Re s > 1$, the lower limit in the antiderivative evaluates to $0$. For the upper limit, we apply the transformation formula~\cite[9.131(2)]{gradshteynryzhik07} to write~\eqref{eq:constant-Fourier-coefficient} as
\begin{equation}
  \lim_{z \to 1^-}
  \bigg(
  \frac{\Gamma(1-s)\Gamma(\frac{s}{2})}{4 \Gamma(1-\frac{s}{2}) \log \varepsilon}
  + \frac{z^{s+1} {}_2F_1(1, 1-\frac{s}{2}, 2-s, 1-z^2)}{4 (1-z^2)^s(s-1) \log \varepsilon}
  + \frac{(-2\log z)^{1-s}}{4 (1-s)\log\varepsilon}
  \bigg).
\end{equation}
The first term in the limit is independent of $z$, while the latter terms jointly contribute $0$ to the limit. All together, we find that
\begin{align} \label{eq:Z-even-continuation}
  \Zeven(s)
   & = \frac{q^{s/2}\zeta(s)}{(4 \log \varepsilon)^{s}}
  + \frac{q^{s/2}\Gamma(1-s)\Gamma(\frac{s}{2})}{4 \Gamma(1-\frac{s}{2}) \log \varepsilon} \\
   & \quad
  + q^{\frac{s}{2}}\sum_{m \neq 0} \bigg(
  \frac{\Gamma(1-s) \Gamma(\frac{s}{2} - \frac{\pi i m}{2 \log \varepsilon})}
  {4\Gamma(1-\frac{s}{2} - \frac{\pi i m}{2 \log \varepsilon}) \log \varepsilon}
  - \frac{\Gamma(1-s)e^{\frac{\pi i}{2}(1-s) \mathrm{sgn}(m)}}{(2\pi \vert m \vert)^{1-s}(4 \log \varepsilon)^s} \bigg)
\end{align}
for $\Re s \in (1, 2)$.
Moreover, we remark that the right-hand side
of~\eqref{eq:Z-even-continuation} converges absolutely away from poles for any $\Re s < 2$.
Combined with the obvious initial obvious half-plane of convergence $\Re s >
  0$, this gives meromorphic continuation of $\Zeven$ to $\mathbb{C}$.

This expression simplifies when $\Re s < 0$, as the two parts of the $m$-sum
in~\eqref{eq:Z-even-continuation} converge independently and can be considered separately.
The contribution of the terms involving $\lvert m \rvert^{1-s}$ is
\[
  -q^{\frac{s}{2}} \frac{\Gamma(1-s)(e^{\frac{\pi i}{2}(1-s)} + e^{-\frac{\pi i}{2}(1-s)})}
  {(2\pi)^{1-s}(4 \log \varepsilon)^s} \sum_{m=1}^\infty \frac{1}{m^{1-s}}
  = - \frac{q^{s/2} \zeta(s)}{(4 \log \varepsilon)^s},
\]
by the functional equation of the Riemann zeta function.
This perfectly cancels the first term at right in~\eqref{eq:Z-even-continuation}.
Hence the meromorphic continuation of $\Zeven(s)$ can be identified with
\begin{equation}\label{eq:Z-even-res-lt-zero}
  \Zeven(s)
  = \frac{q^{\frac{s}{2}} \Gamma(1-s)}{4 \log \varepsilon} \sum_{m \in \mathbb{Z}}
  \frac{\Gamma(\frac{s}{2} - \frac{\pi i m}{2 \log \varepsilon})}
  {\Gamma(1-\frac{s}{2} - \frac{\pi i m}{2 \log \varepsilon})}
\end{equation}
in the region $\Re s < 0$.

\begin{theorem}\label{thm:poisson_even}
  The Dirichlet series $\Zeven(s)$ has meromorphic continuation to all
  $\mathbb{C}$.
  When $\Re s > 0$, the defining sequence~\eqref{eq:zeven_def} converges.
  When $\Re s < 2$ (useful when $\Re s = 0$), the continuation is
  given by~\eqref{eq:Z-even-continuation}.
  For $\Re s < 0$, we can write the continuation as
  \begin{equation*}
    \Zeven(s)
    = \frac{q^{\frac{s}{2}} \Gamma(1-s)}{4 \log \varepsilon} \sum_{m \in \mathbb{Z}}
    \frac{\Gamma(\frac{s}{2} - \frac{\pi i m}{2 \log \varepsilon})}
    {\Gamma(1-\frac{s}{2} - \frac{\pi i m}{2 \log \varepsilon})}.
  \end{equation*}
  The polar behavior is visible from this expression: $\Zeven(s)$ has simple
  poles at $s= \frac{\pi i m}{\log \varepsilon} - 2k$ for any $m \in
    \mathbb{Z}$ and integral $k \geq 0$.
  (The poles when $k = 0$ can be recognized through limiting behavior from the
  left.)
\end{theorem}

\begin{remark}
  Unlike $\Zodd$, $\Zeven$ has no trivial zeros.
  Note however that certain values of $\Zeven(s)$ may be computed exactly due to
  simplification of the gamma functions; for example
  \begin{equation}\label{eq:special_value_negone}
    \Zeven(-1)
    = \frac{1}{4 \sqrt{q} \log \varepsilon} \sum_{m \in \mathbb{Z}}
    \frac{\Gamma(-\frac{1}{2} - \frac{\pi i m}{2 \log \varepsilon})}
    {\Gamma(\frac{3}{2} - \frac{\pi i m}{2 \log \varepsilon})}
    = \frac{1+\varepsilon^2}{(1-\varepsilon^2) \sqrt{q}}.
  \end{equation}
  The difference of~\eqref{eq:special_value_negone} and its Galois conjugate is  \begin{equation}
    \frac{1+\varepsilon^2}{(1-\varepsilon^2) \sqrt{q}}
    +
    \frac{1+\overline{\varepsilon}^2}{(1-\overline{\varepsilon}^2) \sqrt{q}}
    =
    \frac{%
      (1 + \varepsilon^2) (1 - \overline{\varepsilon}^2)
      +
      (1 - \varepsilon^2) (1 + \overline{\varepsilon}^2)
    }{
      (1 -  \varepsilon^2) (1 - \overline{\varepsilon}^2) \sqrt{q}
    }
    =
    0,
  \end{equation}
  where we've used that $(\varepsilon \overline{\varepsilon})^2 = 1$ to complete
  the cancellation.
  As $\Zodd(-1) = 0$ trivially, this also gives a closed form for $\Zboth(-1)$.
  Hence $\Zeven(-1)$ and $\Zboth(-1)$ are always rational.

  This generalizes observations from Navas~\cite[\S4]{navas2001fibonacci}, who
  observed that $Z_5^{\textup{odd}}(-1)$ is rational.
  Navas also considered the rationality of $Z_5$ at all other negative
  odd integers, and we remark that his proof of rationality can be adapted for
  the binomial continuation given in Theorem~\ref{thm:binomial-expression}.
\end{remark}

\section{Continuation via Modular Forms}%
\label{sec:modular}

Finally, we explain how $\Zodd$ and $\Zeven$ may be understood via modular forms.
An excellent reference for notions here is~\cite{stromberg_modularformsclassicalapproach}.

Let $r_1(n) = \#\{x \in \Z :  x^2 = n\}$.
Proposition~\ref{prop:pfib} shows that a $\calO_D$ Fibonacci number can
be detected by detecting squares, from which it follows that
\begin{equation}
  \Zodd(s) = \sum_{n \ge 1} \frac{1}{F_D(2n-1)^s}
  = \frac{1}{4} \sum_{n \ge 1} \frac{r_1(n)r_1(qn-4)}{n^{s/2}}.
\end{equation}
When $D \equiv 1 \bmod 4$, we have $q = D$.
Otherwise, we have $q = 4D$ and the last $r_1(\cdot)$ term
in~\eqref{eq:basic_modular} is $r_1(4D - 4) = r_1(D - 1)$ (i.e.\ $4D - 4$ is a
square if and only if $D - 1$ is a square).
Together, this allows us to rewrite $\Zodd(s)$ as
\begin{equation}\label{eq:basic_modular}
  \Zodd(s) = \frac{1}{4} \sum_{n \ge 1} \frac{r_1(n)r_1(Dn-\ell)}{n^{s/2}},
\end{equation}
where $\ell = \ell(D) = 4$ if $D \equiv 1 \bmod 4$ and otherwise $\ell = 1$.

The sequence $\{r_1(n)\}$ gives the coefficients of a modular form and the sum on the right
of~\eqref{eq:basic_modular} is a shifted convolution Dirichlet series.
These are well-studied.
Equation~\eqref{eq:basic_modular} is closely related to the series studied
in~\cite[\S3]{hkldw_3aps}.
Broadly, we obtain the meromorphic continuation for $\Zodd(s)$ via spectral
expansion.

Let $\theta(z) = \sum_{n \ge 0} r_1(n) e^{2 \pi i n z} \in S_{1/2}(\Gamma_0(4))$ denote
the classical theta function, and consider
$V_1(z) = \Im(z)^{1/2} \theta(Dz) \overline{\theta(z)}$.
Then $V_1(z)$ is an automorphic form of weight $0$ for $\Gamma_0(4D)$ with nebentypus
$\chid = \left( \frac{4D}{\cdot} \right)$.
For a matrix $\gamma = \left( \begin{smallmatrix} a & b \\ c & d
\end{smallmatrix} \right) \in \Gamma_0(4D)$, we define $\chid(\gamma) =
\chid(d)$.
For $h \ge 1$, let $P_h(z, s)$ denote the Poincar{\'e} series
\begin{equation} \label{eq:P-h-definition}
  P_h(z, s) \colonequals P_h(z, s; \chid) \colonequals
  \frac{1}{2} \sum_{\gamma \in \Gamma_{\infty} \backslash \Gamma_0(4D)}
  \Im(\gamma z)^s e^{2 \pi i h \gamma z} \chid(\gamma),
\end{equation}
where $\Gamma_\infty \subset \Gamma_0(4D)$ denotes the stabilizer of the cusp at $\infty$.
Then $P_h$ is also an automorphic form of weight $0$ for $\Gamma_0(4D)$ with
nebentypus $\chid$.
By taking the Petersson inner product of $V_1$ with $P_\ell$ (where $\ell =
  \ell(D)$ is $1$ or $4$, as described above), we obtain a
completed version of $\Zodd$:
\begin{equation}\label{eq:spectral_base}
  4 \Gamma(s) \Zodd(2s)
  = (4D \pi)^{s}
  \bigl\langle V_1, P_\ell(\cdot, \sbar + \tfrac{1}{2}) \bigr\rangle.
\end{equation}

The Poincar\'e series $P_\ell(\cdot, s)$ is a cusp form and has meromorphic
continuation to $\mathbb{C}$ (see for
instance~\cite{iwaniecclassical97}).
As $P_\ell$ decays rapidly at the cusps of $\Gamma_0(4D) \backslash
  \mathcal{H}$, this gives an immediate (but unwieldy) meromorphic continuation
of $\Zodd$.

\begin{theorem}\label{thm:modular}
  The Dirichlet series $\Zodd(s)$ has meromorphic continuation to all
  $s \in \mathbb{C}$, given by
  \begin{equation}
    \Zodd(2s)
    =
    \frac{(4D \pi)^{s}}
    {4 \Gamma(s)}
    \bigl\langle V_1, P_\ell(\cdot, \sbar + \tfrac{1}{2}; \chid) \bigr\rangle.
  \end{equation}
\end{theorem}

\subsection{A sketch of how to understand this continuation}

To make sense of this continuation, we can use spectral theory to get a
spectral resolution of $\Zodd$.
As $P_\ell \in L^2(\Gamma_0(4D)\backslash \mathcal{H}; \chid)$, we can directly expand the Poincar\'e
series across its discrete and continuous spectra.
In the rest of this section, we briefly sketch how to understand the
meromorphic behavior from Theorem~\ref{thm:modular}.

Each function in $L^2(\Gamma_0(4D)\backslash \mathcal{H}; \chid)$ has a
spectral expansion in terms of a basis of Maass forms $\{ \mu_j \}$ and a basis
of Eisenstein series $\{ \mathcal{E}_\mathfrak{a} \}$ (cf.~\cite[Prop. 4.1,
  4.2]{dfi}).
For the Poincar\'e series $P_\ell$, this takes the form
\begin{align}
  \label{eq:P4-spectral expansion}
  P_\ell(z, s)
   & =
  \sum_j
  \bigl\langle P_\ell(\cdot, s), \mu_j \bigr\rangle \mu_j(z)
  \\
   & \!\! +
  \sum_{\mathfrak{a}} \frac{1}{4 \pi}
  \int_{-\infty}^{\infty}
  \bigl\langle
  P_\ell(\cdot, s),
  \mathcal{E}_\mathfrak{a}(\cdot, \tfrac{1}{2} + it)
  \bigr\rangle
  \mathcal{E}_\mathfrak{a}(z, \tfrac{1}{2} + it) dt,
\end{align}
where $\mathfrak{a}$ ranges over cusps of $\Gamma_0(4D)$ that are singular with
respect to $\chid$.

Before inserting this expansion into~\eqref{eq:spectral_base},
it is easier to work with $V_1$ after subtracting an Eisenstein series $E(z)$
that cancels the growth at the cusps.
We call this function $V(z) = V_1(z) - E(z)$.
Now inserting $V_1(z) = V(z) + E(z)$ into~\eqref{eq:spectral_base}, we
find
\begin{equation} \label{eq:regularized-expansion}
  4 \Gamma(s) \Zodd(2s)
  =
  (4 \pi D)^s \bigl\langle V, P_\ell(\cdot, \overline{s}+\tfrac{1}{2}) \bigr\rangle
  +
  (4 \pi D)^s \bigl\langle E, P_\ell(\cdot, \overline{s}+\tfrac{1}{2}) \bigr\rangle.
\end{equation}
For $\Re s$ sufficiently large, we may now use the spectral expansion~\eqref{eq:P4-spectral expansion} to produce the following spectral expansion for $\Zodd(2s)$:
\begin{align}
  \frac{4 \Gamma(s) \Zodd(2s)}{(4\pi D)^s}
   &
  =
  \bigl\langle E, P_\ell(\cdot, \overline{s}+\tfrac{1}{2}) \bigr\rangle
  +
  \sum_j \langle P_\ell(\cdot, \overline{s}+\tfrac{1}{2}), \mu_j \rangle
  \langle V, \mu_j \rangle
  \\
   & \quad
  + \sum_{\mathfrak{a}} \frac{1}{4 \pi}
  \int_{-\infty}^{\infty}
  \bigl\langle
  P_\ell(\cdot, \overline{s}+\tfrac{1}{2}),
  \mathcal{E}_\mathfrak{a}
  \bigr\rangle
  \bigl\langle
  V,
  \mathcal{E}_\mathfrak{a}(\cdot, \tfrac{1}{2} + it)
  \bigr\rangle
  dt.
\end{align}
Following~\cite{hkldw_3aps}, the inner products $\langle V, \mu_j \rangle$ and
$\langle V, \mathcal{E}_\mathfrak{a} \rangle$ can be explicitly computed and
the remaining inner products can be bounded.
This is enough to show that $\Zodd$ has a meromorphic continuation in terms of
the spectrum on $L^2(\Gamma_0(4D) \backslash \mathcal{H} ; \chid)$.

It is not obvious from this expansion whether the meromorphic continuation that
results from spectral resolution resembles the continuations from direct
binomial expansion (Theorem~\ref{thm:binomial-expression}) or Poisson summation
(Theorem~\ref{thm:continuation-Poisson-odd}).
In a sequel paper, we will show that the continuous spectrum and almost all of
the discrete spectrum vanishes, leaving only contributions from dihedral forms.
These components are explicitly computable.
Surprisingly, carrying out these computations ultimately gives the exact same
meromorphic continuation as from Poisson summation, stated in
Theorem~\ref{thm:continuation-Poisson-odd}.
See forthcoming work~\cite{akldwFibonacciModular} for more.

\begin{remark}
  Identical reasoning gives the series for $\Zeven$:
  \begin{equation}
    \Zeven(s) = \frac{1}{4} \sum_{n \ge 1} \frac{r_1(n)r_1(Dn+\ell)}{n^{s/2}}
  \end{equation}
  and the continuation
  $
    4 \Gamma(s) \Zeven(2s)
    = (4D \pi)^{s}
    \bigl\langle V_1, P_{-\ell}(\cdot, \sbar + \tfrac{1}{2})
    \bigr\rangle
  $. Although superficially similar, it is necessary to incorporate work
  from~\cite[\S2]{hoffsteinhulse13} to make sense of the spectral
  decomposition.
  This mirrors the additional technical hurdles in deducing the
  continuation of $\Zeven$ using Poisson summation, and again the continuation
  from modular forms perfectly matches the continuation from (even) Poisson
  summation (Theorem~\ref{thm:poisson_even}).

  It is not apparent to the authors why the meromorphic continuations coming
  from modular forms perfectly match those coming from Poisson summation.
\end{remark}

\bibliographystyle{alpha}
\bibliography{bibfile}

\end{document}